\documentclass{article}
\usepackage{amsmath,amssymb,amsfonts,amsthm}
 \newtheorem{thm}{Theorem}[section]

 \theoremstyle{definition}
 
 \theoremstyle{remark}

 \numberwithin{equation}{section}

\newcommand{\wh}{\widehat}
\newcommand{\wt}{\widetilde}
\newcommand{\wc}{\check}
\newcommand{\D}{\mathbb D}

\title{Generalization of the Brauer Theorem to \\ Matrix Polynomials and Matrix Laurent Series\thanks{Work supported by Gruppo Nazionale di Calcolo Scientifico (GNCS) of INdAM, and by the PRA project
``Mathematical models and computational methods for complex networks'' funded by the University of Pisa.}}

\author{D.A. Bini, B. Meini\\
Dipartimento di Matematica, Universit\`a di Pisa\\ Largo Bruno Pontecorvo, 5, 56127 Pisa, Italy\\
      {\tt bini, meini@dm.unipi.it}}
\begin{document}

\maketitle

\begin{abstract}
Given a square matrix $A$, Brauer's theorem [Duke Math. J. 19 (1952),
  75--91] shows how to modify one single eigenvalue of $A$ via a
rank-one perturbation, without changing any of the remaining
eigenvalues.  We reformulate Brauer's theorem in functional form and
provide extensions to matrix polynomials and to matrix Laurent
series $A(z)$ together with generalizations to shifting a set of
eigenvalues. We provide conditions under which the modified function
$\wt A(z)$ has a canonical factorization $\wt A(z)=\wt U(z)\wt L(z^{-1})$
and we provide explicit expressions of the factors $\wt U(z)$ and $\wt
L(z)$. Similar conditions and expressions are given for the
factorization of $\wt A(z^{-1})$. Some applications are discussed.
\end{abstract}

\section{Introduction}
Brauer's theorem \cite{brauer} relates the eigenvalues of an $n\times
n$ matrix $A$ to the eigenvalues of the modified matrix $\wt A=A+
uv^*$ when either $u$ or $v$ coincides with a right or with a left
eigenvector of $A$. Its original formulation can be stated in the
following form.

\begin{thm}\label{th:brauer}
Let $A$ be an $n\times n$ complex matrix with eigenvalues
$\lambda_1,\ldots,\lambda_n$.  Let $u_k$ be an eigenvector of $A$
associated with the eigenvalue $\lambda_k$, and let $v$ be any
$n$-dimensional vector. Then the matrix $\wt A=A + u_k v^*$ has eigenvalues
$\lambda_1,\ldots,\lambda_{k-1},\lambda_k+v^*u_k,\lambda_{k+1},\ldots,\lambda_n$.
\end{thm}

This elementary result has been applied in different contexts, more
specifically to the spectral analysis of the PageRank matrix
\cite{serra}, to the nonnegative inverse eigenvalue problem
\cite{trabajo}, \cite{soto-rojo}, to accelerating the convergence of
iterative methods for solving matrix equations in particular quadratic
equations \cite{meini:ilas10}, equations related to QBD Markov chains
\cite{hmr01}, \cite{guo03}, to M/G/1-type Markov chains
\cite{bini-meini-spitkovsky}, \cite{blm:book}, or for algebraic
Riccati equations in \cite{gh:shift}, \cite{gim:shift},
\cite{ip:shift}, and \cite{bim:book}.

In this paper we revisit Brauer's theorem in functional form and
generalize it to a non-degenerate $n\times n$ matrix function $A(z)$ analytic in
the open annulus $\mathbb A_{r_1,r_2}=\{z\in\mathbb C:\quad
r_1<|z|<r_2\}$, where $0\le r_1<r_2$, so that it can be described by a
matrix Laurent series $A(z)=\sum_{i\in\mathbb Z}z^iA_i$ convergent for
$z\in\mathbb A_{r_1,r_2}$. Here, non-degenerate means that $\det A(z)$
is not identically zero. Throughout the paper we assume that all the
matrix functions are nondegenerate.

Let $\lambda\in\mathbb A_{r_1,r_2}$ be an eigenvalue of $A(z)$, i.e.,
such that $\det A(\lambda)=0$ and $u\in\mathbb C^n$, $u\ne 0$, be
a corresponding eigenvector, i.e., such that
$A(\lambda)u=0$.
We prove that for any $v\in\mathbb C^n$ such that $v^*u=1$  and for any
$\mu\in\mathbb C$, the function $\wt
A(z)=A(z)(I+\frac{\lambda-\mu}{z-\lambda} uv^*)$ is still analytic in
$\mathbb A_{r_1,r_2}$, where $v^*$ denotes the transpose conjugate of $v$. 
Moreover, $\wt A(z)$ has all the eigenvalues of
$A(z)$, except for $\lambda$ which, if $\mu\in\mathbb A_{r_1,r_2}$, is
replaced by $\mu$, otherwise is removed.  We provide explicit
expressions which relate the coefficients $\wt A_i$ of $\wt
A(z)=\sum_{i\in\mathbb Z}z^i\wt A_i$ to those of $A(z)$.

This transformation, called {\em right shift}, is obtained by using a
right eigenvector $u$ of $A(z)$. A similar version, called {\em left
  shift}, is given by using a left eigenvector $v$ of
$A(z)$. Similarly, a combination of the left and the right shift,
leads to a transformation called {\em double shift}, which enables one
to shift a pair of eigenvalues by operating to the right and to the
left of $A(z)$.

This result, restricted to a matrix Laurent polynomial
$A(z)=\sum_{-h}^k z^i A_i$, provides a shifted function which is still
a matrix Laurent polynomial $\wt A(z)=\sum_{i=-h}^kz^i \wt A_i$. In
the particular case where $A(z)=zI-A$, we find that $\wt A(z)=zI-\wt
A$ is such that $\wt A$ is the matrix given in the classical theorem
of Brauer \cite{brauer}. Therefore, these results provide an extension
of Brauer's Theorem \ref{th:brauer} to matrix Laurent series and to
matrix polynomials.

A further generalization is obtained by performing the right, left, or
double shift simultaneously to a packet of eigenvalues once we are
given an invariant (right or left) subspace associated with this set
of eigenvalues.

Assume that $A(z)$ admits a canonical factorization $A(z)=U(z)L(z^{-1})$,
where $U(z)=\sum_{i=0}^\infty z^iU_i$, $L(z)=\sum_{i=0}^\infty z^iL_i$
are analytic and nonsingular for $|z|\le 1$.  An interesting issue is
to find out under which assumptions the function $\wt A(z)$, obtained
after any kind of the above transformations, still admits a canonical
factorization $\wt A(z)=\wt U(z)\wt L(z^{-1})$.
We investigate this issue and provide explicit expressions to the
functions $\wt U(z)$ and $\wt L(z)$. We examine also the more
complicated problem to determine, if they exist, the canonical
factorizations of both the functions $\wt A(z)$ and $\wt A(z^{-1})$.

These latter results find an interesting application to determine the
explicit solutions of matrix difference equations, by relying on the
theory of standard triples of \cite{glr82}. In fact, in the case where
the matrix polynomial $A(z)=\sum_{i=0}^dz^iA_i$, which defines the
matrix difference equation, has some multiple eigenvalues, the tool of
standard triples cannot be always applied. However, by shifting away
the multiple eigenvalues into a set of pairwise different eigenvalues,
we can transform $A(z)$ into a new matrix polynomial $\wt A(z)$, which
has all simple eigenvalues. This way, the matrix difference equation associated
with $\wt A(z)$ can be solved with the tool of standard triples
and we can easily reconstruct the solution of the original matrix difference
equation, associated with $A(z)$, from that associated with $\wt A(z)$.
This fact enable us to provide formal solutions to the Poisson problem
for QBD Markov chains also in the null-recurrent case
\cite{blm:shift}.

Another important issue related to the shift technique concerns
solving matrix equations. In fact, for quadratic matrix polynomials,
computing the canonical factorization corresponds to solving a pair of
matrix equations. The same property holds, to a certain extent, also
for general matrix polynomials or for certain matrix Laurent series
\cite{blm:book}. If the matrix polynomial has a repeated eigenvalue,
as it happens in null recurrent stochastic processes, then the
convergence of numerical algorithms slows down and the conditioning of
the sought solution deteriorates. Shifting multiple eigenvalues
provides matrix functions, where the associated computational problem
is much easier to solve and where the solution is better
conditioned. Moreover the solution of the original equation can be
obtained by means of the formulas, that we provide in this paper,
relating the canonical factorizations of $A(z)$, $A(z^{-1})$ to the
canonical factorizations of $\wt A(z)$ and $\wt A(z^{-1})$.

An interesting application concerns manipulating a matrix polynomial
$A(z)$ having a singular leading coefficient. In fact, in this case
$A(z)$ has some eigenvalues at infinity. By using the shift technique
it is simple to construct a new matrix polynomial $\wt A(z)$ having
the same eigenvalues of $A(z)$ except for the eigenvalues at infinity
which are replaced by some finite value, say 1. A nonsingular leading
coefficient allows one to apply specific algorithms that require this
condition, see for instance \cite{bini-robol}, \cite{ttz1},
\cite{ttz2}.

The paper is organized as follows. Section \ref{sec:gen} provides a
functional formulation of Brauer's theorem, followed by its extension
to matrix polynomials and to matrix Laurent series. Section
\ref{sec:set} describes a formulation where a set of selected
eigenvalues is shifted somewhere in the complex plane.  Section
\ref{sec:wh} deals with the transformations that the canonical
factorization of a matrix function has after that some eigenvalue has
been shifted away.  Section \ref{sec:specific} deals with some
specific cases and applications. In particular, we examine the
possibility of shifting eigenvalues from/to infinity, we analyze right
and left shifts together with their combination, that is, the double
shift, and apply them to manipulating palindromic matrix
polynomials. Then we analyze quadratic matrix polynomials and the
associated matrix equations. In particular, we relate the solutions of
the shifted equations to the ones of the original equations.

\section{Generalizations and extensions}\label{sec:gen}
In this section, we provide some generalizations and extensions of
Brauer Theorem~\ref{th:brauer}. In particular, we give a functional
formulation, from which we obtain an extension to matrix polynomials
and to matrix (Laurent) power series.

\subsection{Functional formulation and extension to matrix polynomials}

Let $A$ be an $n\times n$ matrix, let $\lambda$ and $u\in\mathbb C^n$,
$u\ne 0$, be such that $Au=\lambda u$.  Consider the rational function
$I+\frac{\lambda-\mu}{z-\lambda}Q$, where $Q=uv^*$ and $v$ is any
nonzero vector such that $v^*u=1$ and $\mu$ is any complex number.  We have
the following

\begin{thm}
The rational function $\wt
A(z)=(zI-A)(I+\frac{\lambda-\mu}{z-\lambda}Q)$ coincides with the
matrix polynomial $\wt A(z)=zI-\wt A$, $\wt
A=A+(\mu-\lambda)Q$. Moreover the eigenvalues of $\wt A$ coincide with
the eigenvalues of $A$ except for $\lambda$ which is replaced by
$\mu$.
\begin{proof}
We have $\wt A(z)=(zI-A)(I+\frac{\lambda-\mu}{z-\lambda}Q)=
zI-A+\frac{\lambda-\mu}{z-\lambda}(zI-A)uv^*=
zI-A+\frac{\lambda-\mu}{z-\lambda}(z-\lambda)uv^*=zI-(A-(\lambda-\mu)Q)$. Moreover,
$\det(zI-\wt
A)=\det(zI-A)\det(I+\frac{\lambda-\mu}{z-\lambda}Q)=\det(zI-A)(1+\frac{\lambda-\mu}{z-\lambda}v^*u)=\det(zI-A)\frac{z-\mu}{z-\lambda}$. This
implies that the eigenvalue $\lambda$ of $A$ is replaced by $\mu$.
\end{proof}
\end{thm}

Relying on the above functional formulation of Brauer theorem, we may
prove the following extension to matrix polynomials:

\begin{thm}\label{th:matpol}
Let $A(z)=\sum_{i=0}^dz^iA_i$ be an $n\times n$ matrix polynomial. Let
$\lambda\in\mathbb C$ and $u\in\mathbb C^n$, $u\ne 0$, be such that
$A(\lambda)u=0$.  For any $\mu\in\mathbb C$ and for any $v\in\mathbb
C^n$ such that $v^*u=1$, the rational function $\wt
A(z)=A(z)(I+\frac{\lambda-\mu}{z-\lambda}Q)$, with $Q=uv^*$, coincides
with the matrix polynomial $ \wt A(z)=\sum_{i=0}^d z^i\wt A_i $, where
\[
\wt A_i=A_i+
(\lambda-\mu)\sum_{k=0}^{d-i-1}\lambda^kA_{k+i+1}Q,~~ i=0,\ldots, d-1,\quad \wt A_d=A_d.
\]
Moreover, the eigenvalues of $\wt A(z)$ coincide with those of $A(z)$
except for $\lambda$ which is shifted to $\mu$.
\end{thm}
\begin{proof}
Since $A(\lambda)u=0$ we find that $\wt
A(z)=A(z)(I+\frac{\lambda-\mu}{z-\lambda}Q)=
A(z)+\frac{\lambda-\mu}{z-\lambda}(A(z)-A(\lambda))uv^*$. Moreover,
since $A(z)-A(\lambda)=\sum_{i=1}^d(z^i-\lambda^i)A_i$ and
$z^i-\lambda^i=(z-\lambda)\sum_{j=0}^{i-1}\lambda^{i-j-1}z^{j}$ we get
\[
\wt A(z)=A(z)+(\lambda-\mu)\sum_{i=1}^d\sum_{j=0}^{i-1}\lambda^{i-j-1}z^{j}A_i Q=
A(z)+(\lambda-\mu)\sum_{i=0}^{d-1} z^i\sum_{k=0}^{d-i-1} \lambda^kA_{k+i+1}Q.
\]
Whence we deduce that $\wt
A_i=A_i+(\lambda-\mu)\sum_{k=0}^{d-i-1}\lambda^kA_{k+i+1}Q$, for
$i=0,\ldots, d-1$. Concerning the eigenvalues, by taking determinants
in both sides of $\wt A(z)=A(z)(I+\frac{\lambda-\mu}{z-\lambda}Q)$,
one obtains $\det\wt A(z)=\det
A(z)\det(I+\frac{\lambda-\mu}{z-\lambda}v^*u)=\det
A(z)\frac{z-\mu}{z-\lambda}$, which completes the proof.
\end{proof}

It is interesting to observe that the function obtained by applying
the shift to a matrix polynomial is still a matrix polynomial with the
same degree.

A different way of proving this result for a matrix polynomial
consists in applying the following three steps: 1) to extend Brauer's
theorem to a linear pencil, that is, to a function of the kind $zB+C$;
2) to apply this extension to the $nd\times nd$ linear pencil
$\mathcal A(z)=z\mathcal B+\mathcal C$ obtained by means of a
Frobenius companion linearization of the matrix polynomial; in fact,
it is possible to choose the vector $v$ in such a way that the shifted
pencil still keeps the Frobenius structure; 3) reformulate the problem
in terms of matrix polynomial and get the extension of Brauer's
theorem. In this approach we have to rely on different companion forms
if we want to apply either the right or the left shift.

This approach, based on companion linearizations, does not work in the
more general case of matrix Laurent series, while the functional
formulation still applies as explained in the next section.

As an example of application of this extension consider the quadratic
matrix polynomial $A(z)=A_0+zA_1+z^2A_2$ where
\[
A_0=-\begin{bmatrix}1&1\\0&1\end{bmatrix},\quad A_1=\begin{bmatrix}4&3\\1&4\end{bmatrix},\quad A_2=-\begin{bmatrix}3&0\\1&2\end{bmatrix}.
\]
Its eigenvalues are $1/3,1/2,1,1$. Moreover $u=[1,0]^*$ is an eigenvector corresponding to the eigenvalue 1.
Applying the Matlab function {\tt polyeig} to $A(z)$ provides the following approximations to the eigenvalues
\begin{verbatim}
0.333333333333333,  0.500000000000002,
0.999999974684400,  1.000000025315599.
\end{verbatim}
The large relative error in the approximation of the eigenvalue 1 is due to the ill conditioning of this double eigenvalue. Performing a shift of $\lambda=1$ 
to $\mu=0$, with $v=u$, so that $Q=\left[\begin{smallmatrix}1&0\\ 0&0\end{smallmatrix}\right]$, one obtains
\[
\wt A_0=-\begin{bmatrix}0&1\\0&1\end{bmatrix},\quad \wt A_1=\begin{bmatrix}1&3\\0&4\end{bmatrix},\quad \wt A_2=-\begin{bmatrix}3&0\\1&2\end{bmatrix}.
\]
Applying the Matlab function {\tt polyeig} to $\wt A(z)$ provides the following approximations to the eigenvalues
\begin{verbatim}
0.000000000000000,  0.333333333333333,
0.500000000000000,  1.000000000000000.
\end{verbatim}
In this case, the approximations are accurate to the machine precision
since the new polynomial has simple well conditioned eigenvalues. This
transformation has required the knowledge of the double eigenvalue and
of the corresponding eigenvector. For many problems, typically for
matrix polynomials encountered in the analysis of stochastic
processes, this information is readily available.

\subsection{The case of matrix (Laurent) power series}
The shift based on the functional approach, described for a matrix
polynomial, can be applied to deal with nondegenerate matrix functions
which are analytic on a certain region of the complex plane.  Let us
start with the case of a matrix power series.

\begin{thm}\label{th:matpower}
Let $A(z)=\sum_{i=0}^\infty z^iA_i$ be an $n\times n$ matrix power
series which is analytic in the disk $\D_r=\{z\in\mathbb C:
|z|<r\}$. Let $\lambda\in \D_r$ and $u\in\mathbb C^n$, $u\ne 0$, be
such that $A(\lambda)u=0$.  For any any $\mu\in \mathbb C$ and
$v\in\mathbb C^n$ such that $v^*u=1$, the function $\wt
A(z)=A(z)(I+\frac{\lambda-\mu}{z-\lambda}Q)$, with $Q=uv^*$, is
analytic for $z\in \D_r$ and coincides with the matrix power series
\[
\wt A(z)=\sum_{i=0}^\infty z^i \wt A_i,\quad \wt A_i=A_i+
(\lambda-\mu)\sum_{k=0}^{\infty}\lambda^kA_{k+i+1},\quad i=0,1\ldots .
\]
Moreover, $\det\wt A(z)=\det A(z)\frac{z-\mu}{z-\lambda}$.
\end{thm}

\begin{proof}
Since $A(\lambda)u=0$, we find that $\wt A(z)=A(z)(I+\frac{\lambda-\mu}{z-\lambda}Q)=
A(z)+\frac{\lambda-\mu}{z-\lambda}(A(z)-A(\lambda))uv^*$. Moreover,
since $A(z)-A(\lambda)=\sum_{i=1}^\infty(z^i-\lambda^i)A_i$ and $z^i-\lambda^i=(z-\lambda)\sum_{j=0}^{i-1}\lambda^{i-j-1}z^{j}$ we get 
\begin{equation}\label{eq:atproof}
\wt A(z)=A(z)+(\lambda-\mu)\sum_{i=1}^\infty\sum_{j=0}^{i-1}\lambda^{i-j-1}z^{j}A_i uv^*,
\end{equation}
provided that the series
$S(z)=\sum_{i=1}^\infty\sum_{j=0}^{i-1}\lambda^{i-j-1}z^{j}A_i$ is
convergent in $\D_r$. We prove this latter property.  Since $A(z)$ is
convergent in $\D_r$, then for any $0<\sigma<r$ there exists a
positive matrix $\Gamma$ such that $| A_i|\le \Gamma\sigma^{-i}$
\cite[Theorem 4.4c]{henrici:book}, \cite[Theorem 3.6]{blm:book}, where
$|A|$ denotes the matrix whose entries are the absolute values of the
entries of $A$ and the inequality holds component-wise.  Let $z$ be
such that $|\lambda|\le |z| <r$ and let $\sigma$ be such that
$|z|<\sigma<r$.  We have
$\sum_{i=1}^\infty\sum_{j=0}^{i-1}|\lambda^{i-j-1}z^{j}|\,|A_i|\le
\sum_{i=1}^\infty\sum_{j=0}^{i-1}|z|^{i-1}|A_i|\le \sigma^{-1}
\sum_{i=1}^\infty i(|z|/\sigma)^{i-1}\Gamma$, and the latter series is
convergent since $0\le |z|/\sigma<1$. Therefore the series $S(z)$ is
absolutely convergent, and hence convergent, for $z \in \D_r$. Since
the series is absolutely convergent, we may exchange the order of the
summations so that we obtain $S(z)=\sum_{i=0}^{\infty}
z^i\sum_{k=0}^{\infty} \lambda^kA_{k+i+1}$.  Whence, from
\eqref{eq:atproof}, we deduce that $\wt
A_i=A_i+(\lambda-\mu)\sum_{k=0}^{\infty}\lambda^kA_{k+i+1}$, for
$i=0,1,\ldots$. By taking determinants in both sides of $\wt
A(z)=A(z)(I+\frac{\lambda-\mu}{z-\lambda}Q)$, one obtains $\det\wt
A(z)=\det A(z)\det(I+\frac{\lambda-\mu}{z-\lambda}v^*u)=\det
A(z)(1+\frac{\lambda-\mu}{z-\mu})=\det A(z)\frac{z-\mu}{z-\lambda}$,
which completes the proof.
\end{proof}

We observe that, if in the above theorem $\mu$ does not belong to
$\D_r$, the eigenvalues of the function $\wt A(z)$ coincide with the
eigenvalues of $A(z)$, together with their multiplicities, except for
$\lambda$. More specifically, if $\lambda$ is a simple eigenvalue of
$A(z)$, then it is not any more eigenvalue of $\wt A(z)$; if $\lambda$
is an eigenvalue of multiplicity $\ell>1$ for $A(z)$, then $\lambda$
is an eigenvalue of multiplicity $\ell-1$ for $\wt A(z)$.  On the
other hand, if $\mu\in\mathbb D_r$ then in $\wt A(z)$ the eigenvalue
$\lambda$ is replaced by $\mu$.

A similar result can be stated for matrix Laurent series, i.e.,
matrix functions of the form $A(z)=\sum_{i\in\mathbb Z}z^iA_i$ which
are analytic over $\mathbb A_{r_1,r_2}=\{z\in\mathbb C:\quad
r_1<|z|<r_2\}$ for $0<r_1<r_2$.

\begin{thm}\label{th:matlaurent}
Let $A(z)=\sum_{i\in\mathbb Z} z^iA_i$ be an $n\times n$  matrix Laurent series which is analytic in the annulus
$\mathbb A_{r_1,r_2}$.
 Let $\lambda\in \mathbb A_{r_1,r_2}$ and $u\in\mathbb C^n$, $u\ne 0$, be such that $A(\lambda)u=0$.
For any $v\in\mathbb C^n$ such that $v^*u=1$ and for any $\mu\in \mathbb C$,  the  function
\begin{equation}\label{eq:shifted}
\wt A(z)=A(z)\left(I+\frac{\lambda-\mu}{z-\lambda}Q\right), \quad Q=uv^*,
\end{equation}
 is analytic for $z\in \mathbb A_{r_1,r_2}$ and coincides with the matrix 
Laurent series $\wt A(z)=\sum_{i\in\mathbb Z} z^i \wt A_i$, where
\begin{equation}\label{eq:matlaurent}\begin{split}
&\wt A_i=A_i+(\lambda-\mu)\sum_{k=0}^{\infty}\lambda^kA_{k+i+1}Q,\quad \hbox{ for }i\ge 0,\\
&\wt A_i=A_i-(\lambda-\mu)\sum_{k=0}^{\infty}\lambda^{-k-1}A_{-k+i}Q, \quad \hbox{ for }i<0.
\end{split}
\end{equation}
Moreover, $\det \wt A(z)=\det A(z)\frac{z-\mu}{z-\lambda}$.
\end{thm}

\begin{proof}
We proceed as in the proof of Theorem \ref{th:matpower}.  Since
$A(\lambda)u=0$, we find that $\wt
A(z)=A(z)(I+\frac{\lambda-\mu}{z-\lambda}Q)=
A(z)+\frac{\lambda-\mu}{z-\lambda}(A(z)-A(\lambda))uv^*$.  Observe
that
$A(z)-A(\lambda)=\sum_{i=1}^\infty(z^i-\lambda^i)A_i+\sum_{i=1}^\infty
(z^{-i}-\lambda^{-i})A_{-i}$.  Moreover, we have
$z^i-\lambda^i=(z-\lambda)\sum_{j=0}^{i-1}\lambda^{i-j-1}z^{j}$ and
$z^{-i}-\lambda^{-i}=-\frac
1{z\lambda}(z-\lambda)\sum_{j=0}^{i-1}\lambda^{-(i-j-1)}z^{-j}$. Thus
we get
\begin{equation}\label{eq:atproof2}
\begin{split}
&\wt A(z)=A(z)+(\lambda-\mu)S_+(z)uv^*-\frac {\lambda-\mu}{z\lambda}S_-(z)uv^*,\\
&S_+(z)=\sum_{i=1}^\infty\sum_{j=0}^{i-1}\lambda^{i-j-1}z^{j}A_i,\quad
S_-(z)=\sum_{i=1}^\infty\sum_{j=0}^{i-1}\lambda^{-(i-j-1)}z^{-j}A_{-i},
\end{split}
\end{equation}
provided that the series $S_+(z)$ and $S_-(z)$ are convergent in
$\mathbb A_{r_1,r_2}$. According to \cite[Theorem 3.6]{blm:book}
\cite[Theorem 4.4c]{henrici:book}, for any $r_1<\sigma<r_2$ there
exists a positive matrix $\Gamma$ such that $| A_i|\le
\Gamma\sigma^{-i}$ for $i\in\mathbb Z$.  Thus we may proceed as in the
proof of Theorem \ref{th:matpower} and conclude that the series
$S_+(z)$ and $S_-(z)$ are absolutely convergent, so that we may
exchange the order of the summations and get
\[
S_+(z)=\sum_{i=0}^{\infty} z^i\sum_{k=0}^{\infty} \lambda^kA_{k+i+1},\quad
S_-(z)=\sum_{i=0}^{\infty} z^{-i}\sum_{k=0}^{\infty} \lambda^{-k}A_{-(k+i+1)}.
\]
Whence from the first equation in \eqref{eq:atproof2} we deduce that
\[\begin{split}
&\wt A_i=A_i+(\lambda-\mu)\sum_{k=0}^{\infty}\lambda^kA_{k+i+1}uv^*,\quad \hbox{ for }i\ge 0,\\
&\wt A_i=A_i-(\lambda-\mu)\sum_{k=0}^{\infty}\lambda^{-k-1}A_{-k+i}uv^*, \quad \hbox{ for }i<0.
\end{split}
\]
 By taking determinants in both sides of $\wt
 A(z)=A(z)(I+\frac{\lambda-\mu}{z-\lambda}Q)$, one obtains $\det\wt
 A(z)=\det A(z)\det(I+\frac{\lambda-\mu}{z-\lambda}v^*u)=\det
 A(z)(1+\frac{\lambda-\mu}{z-\mu})=\det A(z)\frac{z-\mu}{z-\lambda}$,
 which completes the proof.
\end{proof}

Observe that, if $\mu\not\in\mathbb A_{r_1,r_2}$ then the eigenvalues
of $\wt A(z)$ coincide with those of $A(z)$ with their multiplicities
except for $\lambda$.  More specifically, if $\lambda$ is a simple
eigenvalue of $A(z)$, then it is not any more eigenvalue of $\wt
A(z)$; if $\lambda$ is an eigenvalue of multiplicity $\ell>1$ for
$A(z)$, then $\lambda$ is an eigenvalue of multiplicity $\ell-1$ for
$\wt A(z)$.  On the other hand, if $\mu\in\mathbb A_{r_1,r_2}$ then in
$\wt A(z)$ the eigenvalue $\lambda$ is replaced by $\mu$.

Observe also that if $A(z)=\sum_{i=-h}^k z^i A_i$ then from
\eqref{eq:matlaurent} it follows that $\wt A(z)= \sum_{i=-h}^k z^i \wt
A_i$.

\section{Shifting a set of eigenvalues}\label{sec:set}
In this section we provide a generalization of Brauer theorem, where
we can simultaneously shift a set of eigenvalues. We will treat the
case of a matrix, of a matrix polynomial and of a matrix Laurent
series.

Let $A\in\mathbb C^{n\times n}$, $U\in\mathbb C^{n\times m}$ and $\Lambda\in\mathbb C^{m\times m}$,
where $m<n$, be such that $U$ has full rank and $AU=U\Lambda$. In other
words, the eigenvalues of $\Lambda$ are a subset of the eigenvalues of
$A$ and the columns of $U$ span an invariant subspace for $A$
corresponding to the eigenvalues of $\Lambda$.  Without loss of
generality we may assume that the eigenvalues of $\Lambda$ are
$\lambda_1,\ldots,\lambda_m$ where $\lambda_i$ for $i=1,\ldots,n$ are
the eigenvalues of $A$.

We have the following generalization of Brauer's theorem

\begin{thm}\label{th:brauergen}
Let $A$ be an $n\times n$  matrix.
Let $U\in\mathbb C^{n\times m}$ and $\Lambda\in\mathbb C^{m\times m}$, where $m<n$, be such that $U$ has full rank and  $AU=U \Lambda$.
Let $V\in\mathbb C^{n\times m}$ be a full rank matrix such that $V^*U=I$. Let $S\in\mathbb C^{m\times m}$ any matrix. Then the rational matrix function 
\[
\wt A(z)=(zI-A)(I+U(zI-\Lambda)^{-1}(\Lambda-S)V^*)
\]
coincides with the linear pencil $zI-\wt A$ where $\wt A=A-U(\Lambda-S)V^*$. Moreover, $\det \wt A(z)=\det(zI-S)\frac{\det A(z)}{\det (zI-\Lambda)}$,
that is, the eigenvalues of $\wt A$ are given by $\mu_1,\ldots,\mu_m,\lambda_{m+1},\ldots,\lambda_n$, where $\mu_1,\ldots,\mu_m$ are the eigenvalues of the matrix $S$.
\end{thm}
\begin{proof}
Denote $X(z)=(zI-\Lambda)^{-1}(\Lambda-S)$. Since $AU=U\Lambda$, we have 
$\wt A(z)=
zI-A+zUX(z)V^*-U\Lambda X(z)V^*=zI-A+U(z I-\Lambda)X(z)V^*$. Whence we get 
$\wt A(z)=zI-(A-U(\Lambda-S)V^*)$. Concerning the eigenvalues, taking determinants we have $\det \wt A(z)=\det A(z)\det(I+UX(z)V^*)=
\det A(z)\det(I+X(z)V^*U)=\det A(z)\det(I+X(z))$. Since 
$I+X(z)=(zI-\Lambda)^{-1}(zI-\Lambda+\Lambda-S)$, we obtain $\det \wt A(z)=\frac{\det A(z)}{\det(zI-\Lambda)}\det(zI-S)$, which completes the proof.
\end{proof}

The above results immediately extends to matrix polynomials. Assume that we are given a matrix $U\in\mathbb C^{n\times m}$ and 
$\Lambda\in\mathbb C^{m\times m}$, where $m<n$, $U$ has full rank and  $\sum_{i=0}^dA_iU\Lambda^i=0$. Then if $w$ is an eigenvector of $\Lambda$ corresponding to the eigenvalue $\lambda$, i.e., $\Lambda w=\lambda w$ then $0=\sum_{i=0}^dA_iU\Lambda^iw=\sum_{i=0}^dA_iU\lambda^iw$, that is,
$A(\lambda)Uw=0$ so that
$\lambda$ is eigenvalue of $A(z)$ and $Uw$ is eigenvector. We have the following result.

\begin{thm}\label{th:brauergenpol}
Let $A(z)=\sum_{i=0}^dz^iA_i$ be an $n\times n$  matrix polynomial.
Let $U\in\mathbb C^{n\times m}$ and $\Lambda\in\mathbb C^{m\times m}$, where $m<n$, be such that $U$ has full rank and  $\sum_{i=0}^dA_iU\Lambda^i=0$.
Let $V\in\mathbb C^{n\times m}$ be a full rank matrix such that $V^*U=I$. Let $S\in\mathbb C^{m\times m}$ be any matrix.
 Then the rational function
\[
\wt A(z)=A(z)(I+U(zI-\Lambda)^{-1}(\Lambda-S)V^*)\]
coincides with the matrix polynomial
$
\wt A(z)=\sum_{i=0}^dz^i \wt A_i$,
\[
 \wt A_i=A_i+
\sum_{k=0}^{d-j-1}A_{k+j+1} U \Lambda^k(\Lambda-S) V^*,\quad j=0,\ldots, d-1,\quad \wt A_d=A_d.
\]
Moreover, $\det \wt A(z)=\det(zI-S)\frac{\det A(z)}{\det
  (zI-\Lambda)}$, that is, the eigenvalues of $\wt A(z)$ coincide with
those of $A(z)$ except for those coinciding with the eigenvalues of
$\Lambda$ which are replaced by the eigenvalues of $S$.
\end{thm}

\begin{proof}
Denote $X(z)=(zI-\Lambda)^{-1}(\Lambda-S)$. Since  $\sum_{i=0}^dA_iU\Lambda^i=0$ we find that 
\[\begin{split}
\wt A(z)&=A(z)(I+UX(z)V^*)=
A(z)+A(z)UX(z)V^*\\
&=A(z)+(A(z)U-\sum_{i=0}^dA_iU\Lambda^i)X(z)V^*.
\end{split}\]
 Moreover,
since $A(z)U-\sum_{i=0}^dA_iU\Lambda^i=\sum_{i=1}^dA_iU(z^iI-\Lambda^i)$ and $z^iI-\Lambda^i=(zI-\Lambda)\sum_{j=0}^{i-1}z^j \Lambda^{i-j-1}$ we get 
\[\begin{split}
\wt A(z)=&A(z)+ \left(\sum_{i=1}^dA_iU  (zI-\Lambda)\sum_{j=0}^{i-1} z^j \Lambda^{i-j-1}\right)  X(z)V^*\\
=&
A(z)+ \left(\sum_{i=1}^dA_iU  \sum_{j=0}^{i-1} z^j \Lambda^{i-j-1}\right)(\Lambda-S)V^*\\
=& A(z)+ \sum_{j=0}^{d-1} z^j\sum_{k=0}^{d-j-1}  A_{k+j+1}U  \Lambda^{k}  (\Lambda-S)V^*.
\end{split}\]
Whence we deduce that 
$\wt A_i=A_i+
\sum_{k=0}^{d-j-1}A_{k+j+1} U \Lambda^k(\Lambda-S) V^*$, $i=0,\ldots,d-1$, $\wt A_d=A_d$.
 Concerning the eigenvalues, taking determinants we have $\det \wt A(z)=\det A(z)\det(I+UX(z)V^*)=
\det A(z)\det(I+X(z)V^*U)=\det A(z)\det(I+X(z))$. Since $I+X(z)=
(zI-\Lambda)^{-1}(zI-\Lambda+\Lambda-S)$, we obtain $\det \wt A(z)=\frac{\det A(z)}{\det(zI-\Lambda)}\det(zI-S)$, which completes the proof.
\end{proof}

Similarly, we may prove the following extensions of the multishift to matrix power series and to matrix Laurent series. Below, we report only the case of a matrix Laurent series. 

\begin{thm}\label{th:brauergenlaur}
Let $A(z)=\sum_{i\in\mathbb Z}z^iA_i$ be a matrix Laurent series
analytic over the annulus $\mathbb A_{r_1,r_2}$ where $r_1<r_2$.  Let
$U\in\mathbb C^{n\times m}$ and $\Lambda\in\mathbb C^{m\times m}$,
where $m<n$, such that $U$ has full rank and $\sum_{i\in\mathbb
  Z}A_iU\Lambda^i=0$.  Let $V\in\mathbb C^{n\times m}$ be a full rank
matrix such that $V^*U=I$. Let $S\in\mathbb C^{m\times m}$ be any
matrix.
 Then the matrix function
\[
\wt A(z)=A(z)(I+U(zI-\Lambda)^{-1}(\Lambda-S)V^*)
\]
 coincides with the matrix Laurent series
$
\wt A(z)=\sum_{i\in\mathbb Z}z^i \wt A_i$,
\[\begin{split}
& \wt A_i=A_i+
\sum_{k=0}^{\infty}A_{k+i+1} U \Lambda^k(\Lambda-S) V^*,\quad i\ge 0,\\
& \wt A_i=A_i-
\sum_{k=0}^{\infty}A_{-k+i} U \Lambda^{-k+1}(\Lambda-S) V^*,\quad i< 0.
\end{split}
\]
Moreover, $\det \wt A(z)=\det(zI-S)\frac{\det A(z)}{\det (zI-\Lambda)}$.
\end{thm}

Observe that if the matrix $S$ in the above theorem has eigenvalues
$\mu_1,\ldots,\mu_m$ where $\mu_1,\ldots,\mu_k\in\mathbb A_{r_1,r_2}$
and $\mu_{k+1},\ldots,\mu_m\not\in\mathbb A_{r_1,r_2}$, then the
matrix function $\wt A(z)$ is still defined and analytic in $\mathbb
A_{r_1,r_2}$. The eigenvalues of $\wt A(z)$ coincide with those of
$A(z)$ except for $\lambda_1,\ldots,\lambda_m$ whose multiplicity is
decreased by 1, moreover, $\wt A(z)$ has the additional eigenvalues
$\mu_1,\ldots,\mu_k\in\mathbb A_{r_1,r_2}$.

\section{Brauer's theorem and canonical factorizations}\label{sec:wh}
Assume we are given a matrix Laurent series $A(z)=\sum_{i\in\mathbb
  Z}z^iA_i$, analytic over $\mathbb A_{r_1,r_2}$, $0<r_1<1<r_2$, which
admits the canonical factorization \cite{bs}
\begin{equation}\label{eq:cf}
A(z)=U(z)L(z^{-1}),\quad 
U(z)=\sum_{i=0}^\infty z^i U_i,\quad L(z)=\sum_{i=0}^\infty z^iL_i,
\end{equation}
where $U(z)$ and $L(z)$ are invertible for $|z|\le1$.  We observe that
from the invertibility of $U(z)$ and $L(z)$ for $|z|\le 1$ it follows
that $\det L_0\ne 0$, $\det U_0\ne 0$. Therefore, the canonical
factorization can be written in the form $A(z)=\wc U(z)K\wc L(z^{-1})$
where $K=U_0L_0$ and $\wc U(z)=U(z)U_0^{-1}$, $\wc L(z)=L_0^{-1}L(z)$.
The factorization in this form is unique \cite{cg:book}.

The goal of this section is to find out if, and under which
conditions, this canonical factorization is preserved under the shift
transformation, and to relate the canonical factorization of $\wt
A(z)$ to that of $A(z)$.

This problem has been recently analyzed in the framework of stochastic
processes in \cite{blm:shift} for a quadratic matrix polynomial and used
to provide an analytic solution of the Poisson problem.

We need a preliminary result whose proof is given in \cite{blm:shift}.

\begin{thm}\label{thm:w}
Let $A(z)=z^{-1}A_{-1}+A_0+zA_1$ be an $n\times n$ matrix Laurent
  polynomial, invertible in the annulus $\mathbb A_{r_1,r_2}$ 
  having the canonical factorization
\[
A(z)=(I-zR_+)K_+(I-z^{-1}G_+).
\]
We have the following properties:
\begin{enumerate}

\item Let $H(z):=A(z)^{-1}$. Then, for $z\in \mathbb A_{r_1,r_2}$ it
  holds that $H(z)=\sum_{i=-\infty}^{+\infty}z^iH_i$ where
\[
H_i=\left\{\begin{array}{ll}
G_+^{-i}H_0,& \hbox{for } i<0,\\
\sum_{j=0}^{+\infty}G_+^j K_+^{-1} R_+^j,&\hbox{for } i=0,\\ 
H_0R_+^i,& \hbox{for } i>0.
\end{array}\right.
\]
\item If $H_0$ is nonsingular, then $A(z^{-1})$ has the canonical
  factorization $A(z^{-1})=(I-z R_-) K_-(I-z^{-1} G_-)$, where
  $K_-=A_0+A_{-1}G_-=A_0+R_- A_{1}$ and $G_-=H_0 R_+ H_0^{-1}$,
  $R_-=H_0^{-1} G_+ H_0$.
\end{enumerate}
\end{thm}

Now we can prove the following result concerning the canonical factorization of $\wt A(z)$.

\begin{thm}\label{th:rightcan}
Assume that $A(z)=\sum_{i\in\mathbb Z}z^iA_i$ admits the canonical
factorization \eqref{eq:cf}. Let $\lambda\in\mathbb A_{r_1,r_2}$ and
$u\in\mathbb{C}^n$, $u\ne 0$, be such that $A(\lambda)u=0$ and
$|\lambda|<1$.  Let $v\in\mathbb C^n$ be such that $v^*u=1$ and let
$\mu\in\mathbb C$ with $|\mu|<1$. Then the matrix Laurent polynomial
$\wt A(z)=A(z)(I+\frac{\lambda-\mu}{z-\lambda}Q)$, where $Q=uv^*$,
admits the following canonical factorization
\[
\wt A(z)=\wt U(z)\wt L(z^{-1})
\]
where $\wt U(z)=U(z)$, $\wt L(z)=\sum_{i=0}^\infty z^i \wt L_i$, with 
\[
\wt L_0=L_0,~~\wt L_i=L_i-(\lambda-\mu) \sum_{j=1}^\infty \lambda^{-j}L_{j+i-1}Q,\quad i\ge 1.
\]
Moreover, 
$\det\wt L(z^{-1})=\det L(z^{-1})\frac{z-\mu}{z-\lambda}$.
\end{thm}
\begin{proof}
Observe that the condition $A(\lambda)u=0$ implies that
$U(\lambda)L(\lambda^{-1})u=0$. Since $|\lambda|<1$ then $U(\lambda)$
is nonsingular, therefore $L(\lambda^{-1})u=0$. Define $\wt
L(z^{-1})=L(z^{-1})(I+\frac{\lambda-\mu}{z-\lambda}Q)$.  In view of
Theorem \ref{th:matlaurent}, the matrix function $\wt L(z^{-1})$ is a
matrix power series in $z^{-1}$, obtained by applying the right shift
to $L(z^{-1})$. Moreover, the expression for its matrix coefficients
follows from \eqref{eq:matlaurent}.  By taking the determinant of $\wt
L(z)$ one
obtains $\det\wt L(z^{-1})=\det L(z^{-1})\det
(I+\frac{\lambda-\mu}{z-\lambda}Q)= \det L(z^{-1})\frac{z-\mu}{z-\lambda}$.  Therefore, since $\det \wt L(z)=\det L(z)\frac{z^{-1}-\mu}{z^{-1}-\lambda}$, the function $\wt L(z)$ is
nonsingular for $|z|\le 1$, so that $\wt A(z)=\wt U(z)\wt L(z^{-1})$ is a canonical factorization.
\end{proof}

The above theorem relates the canonical factorization of $A(z)$ with
the canonical factorization of the shifted function $\wt
A(z)$. Assume we are given a canonical factorization of $A(z^{-1})$,
it is a natural question to figure out if also $\wt A(z^{-1})$
admits a canonical factorization, and if it is related to the one of
$A(z^{-1})$. This issue is motivated by applications in the solution
of the Poisson problem in stochastic models, where both the canonical
factorizations of $\wt A(z)$ and $\wt A(z^{-1})$ are
needed for the existence of the solution and for providing its explicit expression
\cite{bdlm:poisson}.

We give answer to this question in the case of matrix Laurent polynomials of the kind $A(z)=z^{-1}A_{-1}+A_0+zA_1$.

\begin{thm}\label{thm:crspr}
Let $A(z)=z^{-1}A_{-1}+A_0+zA_1$ be analytic in the annulus $\mathbb
A_{r_1,r_2}$.  Assume that $A(z)$ and $A(z^{-1})$ admit the canonical
factorizations
\[
\begin{split}
&A(z)=(I-zR_{+})K_{+}(I-z^{-1}G_{+}),\\
&A(z^{-1})=(I-zR_{-})K_{-}(I-z^{-1}G_{-}).
\end{split}
\]
Let $\lambda\in\mathbb A_{r_1,r_2}$, $|\lambda|<1$, and $u\in\mathbb C^n$, $u\ne 0$, 
be  such that $A(\lambda)u=0$. Let $\mu\in\mathbb C$, $|\mu|<1$, and
$v$ be any vector such that: $v^*u=1$,
$(\lambda-\mu) v^* G_- u \ne 1$.
Set  $Q=uv^*$,  and
define the matrix Laurent polynomial 
$\wt A(z)=z^{-1}\wt A_{-1}+\wt A_0+z\wt A_1$ as in \eqref{eq:shifted}. 
Then $\wt A_{-1}=A_{-1}+(\lambda-\mu)(A_0+\lambda A_1)Q$, $\wt A_0=A_0+(\lambda-\mu)A_1Q$, $\wt A_1=A_1$, moreover
$\wt A(z)$ and $\wt A(z^{-1})$ 
have the canonical factorizations
\begin{equation}\label{slfactpr}
\begin{split}
&\wt A(z)=(I-z\wt R_{+})\wt K_{+}(I-z^{-1}\wt G_{+}),\\
&\wt A(z^{-1})=(I-z\wt R_{-})\wt K_{-}(I-z^{-1}\wt G_{-}),
\end{split}
\end{equation}
where
$\wt R_{+}=R_{+}$, $\wt K_{+}=K_{+}$, $\wt G_{+}=G_{+}+(\mu-\lambda)Q$,
$\wt R_{-}=\wt W^{-1}\wt G_{+} \wt W$, $\wt G_{-}=\wt W R_+ \wt W^{-1}$, 
$\wt K_-= \wt A_{0}+\wt A_{-1}\wt G_{-}= \wt A_0+\wt R_{-} A_{1}$,
  with $\wt W=W+(\mu-\lambda)  QWR_{+}$, $W=\sum_{i=0}^\infty G_+^iK_+^{-1}R_+^i$.
\end{thm}

\begin{proof}
The expressions for the matrix coefficients $\wt A_i$, $i=-1,0,1$, follow from Theorem \ref{th:matlaurent}.
Observe that, since $|\lambda|<1$, from the definition of canonical factorization, it follows that the matrix $I-\lambda R_{+}$ is not singular. Therefore, since $A(\lambda)u=0$ and also $K_+$ is nonsingular, from  
the canonical factorization of $A(z)$, we have $(I-\lambda^{-1}G_{+})u=0$. 
The existence and the expression of the canonical factorization of $\wt A(z)$ can be obtained from Theorem 
\ref{th:rightcan} by setting $U(z)=(I-zR_{+})K_{+}$ and $L(z)=I-zG_{+}$. 
In fact, Theorem 
\ref{th:rightcan} implies that $\wt A(z)$ has the canonical factorization $\wt A(z)=\wt U(z) \wt L(z)$ with $\wt U(z)=U(z)$ and $\wt L(z)=\wt L_0+z \wt L_1$, where $\wt L_0=L_0=I$ and $\wt L_1= L_1-(\lambda-\mu)\lambda^{-1}L_1Q=-G_{+}+(\lambda-\mu)\lambda^{-1}G_{+}Q$. Since $\lambda^{-1}G_{+}Q=Q$, we find that $\wt L(z)=I-z(G_{+}+(\mu-\lambda)Q)$.
To prove that $\wt A(z^{-1})$ has a canonical factorization, we apply Theorem \ref{thm:w}  to $\wt A(z)$.
To this purpose, we show that the matrix $\wt W=\sum_{i=0}^{+\infty}\wt G_{+}^i\wt K_{+}^{-1}\wt R_{+}^i$ is nonsingular.
Observe that $\wt G_{+}^i=G_{+}^i+(\mu-\lambda) QG_{+}^{i-1}$, for $i\ge 1$. Therefore, since $\wt R_{+}=R_{+}$ and $\wt K_{+}=K_{+}$, we may write
\[
\begin{split}
\wt W=&K_{+}^{-1}+\sum_{i=1}^{+\infty}(G_{+}^i+(\mu-\lambda) QG_{+}^{i-1})K_{+}^{-1}R_{+}^i\\
=&K_{+}^{-1}+\sum_{i=1}^{+\infty}G_{+}^iK_{+}^{-1}R_{+}^i+(\mu-\lambda) Q\left(\sum_{i=0}^{+\infty}G_{+}^iK_{+}^{-1}R_{+}^i \right) R_{+}\\
=&W-(\lambda-\mu) QWR_{+}.
\end{split}
\]
By Theorem 3.22 of \cite{blm:book} applied to $A(z)$,  we have
$\det W\ne 0$.
Therefore $\det 
\wt W=\det(I-(\lambda-\mu) QWR_{+}W^{-1})\det W$. Moreover, since $Q=u v^*$, then
the matrix $I-(\lambda-\mu) QWRW^{-1}$ is nonsingular if and only if 
\begin{equation}\label{eq:tmp}
(\lambda-\mu)
v^* WR_{+}W^{-1} u\ne 1.
\end{equation}
   Since  $G_{-}=WR_{+}W^{-1}$, for Theorem \ref{thm:w} where $W=H_0$, 
condition \eqref{eq:tmp}
holds if $(\lambda-\mu)v^* G_{-} u\ne 1$, which is satisfied by assumption. Therefore
 the matrix $\wt W$ is nonsingular, so that
 from Theorem \ref{thm:w} applied to  the matrix Laurent polynomial   $\wt A(z)$, we deduce that $\wt A(z^{-1})$ has the  canonical factorization 
$\wt A(z^{-1}) =(I-z\wt R_{-})\wt K_{-}(I-z^{-1}\wt G_{-})$ with
$\wt K_{-}=\wt A_0+\wt A_{-1}\wt G_{-}=\wt A_0+\wt R_{-} A_{1}$ and $\wt G_{-}=\wt W
  R_{+}  \wt W^{-1}$,  $\wt R_{-}= \wt W^{-1} \wt G_{+} \wt W$. 
\end{proof}

\section{Some specific cases and applications}\label{sec:specific}
\subsection{Left, right, and double shift}
Observe that the Brauer theorem and its extensions can be applied to the matrix $A^T$ or to the matrix polynomial $A(z)^T$. In this case we need to know a left eigenvector or a left invariant subspace of $A(z)$. The result is that we may shift a group of eigenvalues relying on the knowledge of a left invariant subspace. We refer to this procedure as to {\em left shift}.   It is interesting to point out that right shift and left shift can be combined together. This combination is particularly convenient when dealing with palindromic polynomials. In fact, by shifting a group of eigenvalues by means of the right shift and the reciprocals of these eigenvalues by means of a left shift, we may preserve the palindromic structure.

Performing the double shift is also necessary when we are given a canonical factorization of $A(z)$ and we have to shift both an eigenvalue of modulus greater than 1 and an eigenvalue of modulus less than 1. This situation is encountered in the analysis of certain stochastic processes.

We describe the left shift as follows.

\begin{thm}\label{th:ls}
Let $A(z)=\sum_{i\in\mathbb Z}z^i A_i$ be a matrix Laurent series analytic in the annulus $\mathbb A_{r_1,r_2}$.
Let $\lambda\in\mathbb A_{r_1,r_2}$,   be an eigenvalue of $A(z)$ such that  $v^*A(\lambda)=0$ for a nonzero vector $v\in\mathbb C^n$. Given $\mu\in\mathbb C$ define the matrix function
\[
\wt A(z)=\left(I+\frac{\lambda-\mu}{z-\lambda}S\right)A(z),\quad S=yv^*
\]
where $y\in\mathbb C^n$ is such that $v^*y=1$. Then $\wt A(z)$ admits the power series expansion
$\wt A(z)=\sum_{i\in\mathbb Z}z^i\wt A_i $,
convergent for $z\in\mathbb A_{r_1,r_2}$, where
\[\begin{split}
&\wt A_i=A_i+(\lambda-\mu)\sum_{j=0}^{\infty}\lambda^j S A_{i+j+1},\quad \hbox{ for }i\ge 0,\\
&\wt A_i=A_i-(\lambda-\mu)\sum_{j=0}^{\infty}\lambda^{-j-1}S A_{i-j}, \quad \hbox{ for }i<0.
\end{split}
\]
Moreover, we have
$\det \wt A(z)=\det A(z)\frac{z-\mu}{z-\lambda}$. 
\end{thm}
\begin{proof}
It is sufficient to apply Theorem \ref{th:matlaurent} to the matrix function $A(z)^T$.
\end{proof}

We provide a brief description of the double shift where two eigenvalues $\lambda_1$ and $\lambda_2$ are shifted by means of a right and a left shift respectively.
We restrict our attention to the case where $\lambda_1\ne\lambda_2$. The general case requires a more accurate analysis and is not treated here. 

\begin{thm}\label{th:ds}
Let $A(z)=\sum_{i\in\mathbb Z}z^i A_i$ be a matrix Laurent series analytic in the annulus $\mathbb A_{r_1,r_2}$.
Let $\lambda_1,\lambda_2\in\mathbb A_{r_1,r_2}$, $\lambda_1\ne\lambda_2$,  be two eigenvalues of $A(z)$ such that $A(\lambda_1)u=0$, $v^*A(\lambda_2)=0$ for nonzero vectors $u,v\in\mathbb C^n$. Given $\mu_1,\mu_2\in\mathbb C$ define the matrix function
\begin{equation}\label{eq:atds}
\wt A(z)=\left(I+\frac{\lambda_2-\mu_2}{z-\lambda_2}S\right)A(z)\left(I+\frac{\lambda_1-\mu_1}{z-\lambda_1}Q\right),\quad Q=uw^*,\quad S=yv^*
\end{equation}
where $w,y\in\mathbb C^n$ are such that $w^*u=1$, $v^*y=1$. Then $\wt A(z)$ admits the power series expansion
$\wt A(z)=\sum_{i\in\mathbb Z}z^i\wt A_i$,
convergent for $z\in\mathbb A_{r_1,r_2}$, where
\begin{equation}\label{eq:at}\begin{split}
&\wt A_i=\wh A_i+(\lambda_2-\mu_2)\sum_{j=0}^{\infty}\lambda_2^j S \wh A_{i+j+1},\quad \hbox{ for }i\ge 0,\\
&\wt A_i=\wh A_i-(\lambda_2-\mu_2)\sum_{j=0}^{\infty}\lambda_2^{-j-1}S \wh A_{i-j}, \quad \hbox{ for }i<0,
\end{split}
\end{equation}
with
\begin{equation}\label{eq:ah}
\begin{split}
&\wh A_j=A_j+(\lambda_1-\mu_1)\sum_{k=0}^{\infty}\lambda_1^kA_{k+j+1}Q,\quad \hbox{ for }j\ge 0,\\
&\wh A_j=A_j-(\lambda_1-\mu_1)\sum_{k=0}^{\infty}\lambda_1^{-k-1}A_{-k+j}Q, \quad \hbox{ for }j<0.
\end{split}
\end{equation}
Moreover, we have
$\det \wt A(z)=\det A(z)\frac{z-\mu_1}{z-\lambda_1}\frac{z-\mu_2}{z-\lambda_2}$. 
\end{thm}
\begin{proof}
 Denote $\wh A(z)=A(z)(I+\frac{\lambda_1-\mu_1}{z-\lambda_1}Q),\quad Q=uw^*$. In view of Theorem \ref{th:matlaurent}, the matrix function $\wh A(z)$ is analytic in $\mathbb A_{r_1,r_2}$, moreover $\wh A(z)=\sum_{i\in\mathbb Z}z^i\wh A_i$ where $\wh A_i$ are defined in \eqref{eq:ah}. We observe that $v^*\wh A(\lambda_2)=v^*A(\lambda_2)
(I+\frac{\lambda_1-\mu_1}{\lambda_2-\lambda_1}Q)=0$ since $v^*A(\lambda_2)=0$ and 
$\lambda_1\ne\lambda_2$. Thus,  we may apply Theorem \ref{th:ls}
to the matrix function $\wh A(z)$ to shift the eigenvalue 
$\lambda_2$ to $\mu_2$ and get the shifted function $\wt A(z)=\sum_{i\in\mathbb Z}z^i\wt A_i$ where the coefficients $\wt A_i$ are given in \eqref{eq:at}.
The expression for the determinant follows by computing the determinant of both sides of \eqref{eq:atds}.
\end{proof}

In the above theorem, one can prove that applying the left shift followed by the right shift provides the same matrix Laurent series.

Observe that, if $A(z)$ is a matrix Laurent polynomial, i.e., 
$A(z)=\sum_{i=-h}^k z^iA_i$ then $\wt A_i=0$ for $i>k$ and $i<-h$, i.e., also $\wt A(z)$ is a matrix Laurent polynomial. In particular, if $A(z)= z^{-1}A_{-1}+A_0+zA_1$, then $\wt A(z)=\sum_{i=-1}^1 z^i\wt A_i$ with 
\[
\begin{split}
&\wt A_{1}=A_1,\\
& \wt A_0=A_0+(\lambda_1-\mu_1)A_1Q+(\lambda_2-\mu_2)SA_1,\\
&\wt A_{-1}=A_{-1}-(\lambda_1-\mu_1)\lambda_1^{-1}A_{-1}Q-(\lambda_2-\mu_2)\lambda_2^{-1}S(A_{-1}-(\lambda_1-\mu_1)\lambda_1^{-1}A_{-1}Q).
\end{split}
\]

If the matrix function admits a canonical factorization $A(z)=U(z)L(z^{-1})$ then we can prove that also the double shifted function $\wt A(z)$ admits a canonical factorization. This result is synthesized in the next theorem.

\begin{thm}
Under the assumptions of Theorem \ref{th:ds},
if $A(z)$ admits a canonical factorization $A(z)=U(z)L(z^{-1})$, with $U(z)=\sum_{i=0}^\infty z^iU_i$,  $L(z)=\sum_{i=0}^\infty z^iL_i$,
and if $|\lambda_1|,|\mu_1|<1$, $|\lambda_2|, |\mu_2|>1$, then $\wt A(z)$ has the canonical factorization $\wt A(z)=\wt U(z)\wt L(z^{-1})$, with
$\wt U(z)=\sum_{i=0}^\infty z^i\wt U_i$,  $\wt L(z)=\sum_{i=0}^\infty z^i \wt L_i$, where
\[\begin{split}
& \wt L_0=L_0,~~\wt L_i=L_i-(\lambda_1-\mu_1) \sum_{j=1}^\infty \lambda_1^{-j}L_{j+i-1}Q\quad i\ge 1,\\
& \wt U_i=U_i+(\lambda_2-\mu_2)\sum_{j=0}^{\infty}\lambda_2^j S U_{i+j+1}, \quad i\ge 0.
\end{split}
\]
\end{thm}
\begin{proof} Replacing
 $A(z)$ with $U(z)L(z^{-1})$ in equation \eqref{eq:atds} yields the
  factorization $\wt A(z)=\wt U(z)\wt L(z^{-1})$ where $\wt
  U(z)=(I+\frac{\lambda_2-\mu_2}{z-\lambda_2}S)U(z)$, and $\wt
  L(z^{-1})=L(z^{-1})(I+\frac{\lambda_1-\mu_1}{z-\lambda_1}Q)$. Since
  $v^*U(\lambda_2)=0$ and $L(\lambda_1^{-1})u=0$, we have that $\wt
  U(z)$ and $\wt L(z)$ are obtained by applying the left shift and the
  right shift to $U(z)$ and $L(z^{-1})$ respectively. Thus the
  expressions for $\wt U(z)$ and $\wt L(z)$ follow from Theorems
  \ref{th:ls} and \ref{th:matlaurent}, respectively.
\end{proof}

\subsection{Shifting eigenvalues from/to infinity}
Recall that for a matrix polynomial having a singular leading
coefficient, some eigenvalues are at  infinity.  These eigenvalues
may cause numerical difficulties in certain algorithms for the
polynomial eigenvalue problem.  In other situations, having eigenvalues
far from the unit circle, possibly at  infinity, may increase the
convergence speed of certain iterative algorithms.  With the shift
technique it is possible to move eigenvalues from/to  infinity in
order to overcome the difficulties encountered in solving specific
problems.  In this section we examine this possibility.

We restrict our attention to the case of matrix polynomials. Assume we are given $A(z)=\sum_{i=0}^d z^iA_i$ where $A_d$ is singular, and a vector $u\ne 0$ such that $A_du=0$. Consider the reverse polynomial
$A_{\mathcal R}(z)=\sum_{i=0}^d z^iA_{d-i}$,  observe that $A_{\mathcal R}(0)u=A_du=0$ and that the eigenvalues of $A_{\mathcal R}(z)$ are the reciprocals of the eigenvalues of $A(z)$, with the convention that $1/0=\infty$ and $1/\infty=0$.
Apply to $A_{\mathcal R}(z)$ a right shift of $\lambda=0$ to a value $\sigma\ne 0$, and obtain a shifted polynomial $\wt A_{\mathcal R}(z)=A_{\mathcal R}(z)(I-z^{-1}\sigma Q)$, with $Q=uv^*$, $v^*u=1$, having the same eigenvalues of $A_{\mathcal R}(z)$ except for the null eigenvalue which is replaced by $\sigma$. Revert again the coefficients of the polynomial $\wt A_{\mathcal R}(z)$ and obtain a new matrix polynomial $\wt A(z)=\sum_{i=0}^d z^i\wt A_i=z^dA_{\mathcal R}(z^{-1})=A(z)(I-z\sigma Q)$. This polynomial has the same eigenvalues of $A(z)$ except for the eigenvalue at  infinity which is replaced by $\mu=\sigma^{-1}$.
A direct analysis shows that the coefficients of $\wt A(z)$ are given by
\begin{equation}\label{eq:frominf}
\begin{split}
&\wt A_0=A_0,\\
&\wt A_i=A_i-\mu^{-1}A_{i-1}Q,\quad i=1,\ldots,d.
\end{split}
\end{equation}
A similar formula can be obtained for the left shift.

If the kernel of $A_d$ has dimension $k\ge 1$ and we know a basis of the kernel, then we may apply the above shifting technique $k$ times in order to remove $k$  eigenvalues at infinity.  This transformation does not guarantee that the leading coefficient of $\wt A(z)$ is nonsingular. However, if $\det \wt A_d=0$ we may repeat the shift.

An example of this transformation is given by the following quadratic matrix polynomial
\[
A(z)=z^2\begin{bmatrix}
0& 0& 0 \\
0& 1& 0 \\
0& 0& 1 \\
\end{bmatrix}+z
\begin{bmatrix}
0& 1& 1 \\
0& 1& 1 \\
0& 1& 1 \\
\end{bmatrix}
+\begin{bmatrix}
1& 0& -1 \\
1& 2& 0 \\
1& 1& 1
\end{bmatrix},
\]
which has two eigenvalues at infinity, and 4 eigenvalues equal to $\pm {\bf i}\sqrt 3$, $\pm {\bf i}$, where ${\bf i}$ is the imaginary unit.  The Matlab function {\tt polyeig} applied to this polynomial yields the following approximations to the eigenvalues
\begin{verbatim}
                        Inf + 0.000000000000000e+00i
      1.450878662283769e-15 + 1.732050807568876e+00i
      1.450878662283769e-15 - 1.732050807568876e+00i
     -8.679677645252271e-17 + 9.999999999999998e-01i
     -8.679677645252271e-17 - 9.999999999999998e-01i
                        Inf + 0.000000000000000e+00i.
\end{verbatim}
Since $e_1=(1,0,0)^*$ is such that $A_2e_1=0$, we may apply \eqref{eq:frominf} with $\mu=1$ and $Q=e_1e_1^*$
yielding the transformed matrix polynomial
\[
z^2\begin{bmatrix}
0& 0& 0 \\
0& 1& 0 \\
0& 0& 1 \\
\end{bmatrix}+z
\begin{bmatrix}
-1& 1& 1 \\
-1& 1& 1 \\
-1& 1& 1 \\
\end{bmatrix}
+\begin{bmatrix}
1& 0& -1 \\
1& 2& 0 \\
1& 1& 1
\end{bmatrix},
\]
which still has the same eigenvalues of $A(z)$ except one eigenvalue at infinity which is replaced by 1. Since $e_1$ is still in the kernel of the leading coefficient, we may perform again the shift to the above matrix polynomial with $Q=e_1e_1^*$ and $\mu=1/2$ yielding the matrix polynomial
\[
z^2\begin{bmatrix}
2&0&0\\2&1&0\\
2&0&1
\end{bmatrix}+
z\begin{bmatrix}
-3&1&1\\
-3&1&1\\
-3&1&1
\end{bmatrix}+
\begin{bmatrix}
1&0&-1\\
1&2&0\\
1&1&1
\end{bmatrix},
\] 
which has no eigenvalues at infinity.
Applying again {\tt polyeig} to this latter polynomial yields
\begin{verbatim}
      5.000000000000000e-01 + 0.000000000000000e+00i
      9.999999999999993e-01 + 0.000000000000000e+00i
     -1.456365117861229e-16 + 1.732050807568876e+00i
     -1.456365117861229e-16 - 1.732050807568876e+00i
     -4.885759970651636e-17 + 9.999999999999996e-01i
     -4.885759970651636e-17 - 9.999999999999996e-01i.
\end{verbatim}
It is evident that eigenvalues $\pm{\bf i}$ and $\pm {\bf i}\sqrt 3$ are left unchanged, while the two eigenvalues of the original polynomial at  infinity are replaced by 1 and by $1/2$. 

A similar technique can be applied to shift a finite eigenvalue of $A(z)$ to infinity. Assume that $A(\lambda)u=0$ for $u\ne 0$ where $\lambda\ne 0$. Consider the reversed polynomial $A_{\mathcal R}(z)=z^dA(z^{-1})$ and apply the right shift to $A_{\mathcal R}(z)$ to move $\lambda^{-1}$ to $0$.
This way we obtain the polynomial $\wt A_{\mathcal R}(z)=A_{\mathcal R}(z)(I+\frac{\lambda^{-1}}{z-\lambda^{-1}}Q)$, $Q=uv^*$, with $v$ any vector such that $v^*u=1$.
Revert again the coefficients of the polynomial $\wt A_{\mathcal R}(z)$ and obtain a new matrix polynomial $\wt A(z)=\sum_{i=0}^d z^i\wt A_i=z^dA_{\mathcal R}(z^{-1})=A(z)(I+\frac{\lambda^{-1}}{z^{-1}-\lambda^{-1}} Q)$. The equations which relate the coefficients of $A(z)$ and $\wt A(z)$ can be obtained by reverting the coefficients of $\wt A_{\mathcal R}(z)$ and relying on Theorem \ref{th:matpol}. This way we get
\begin{equation}\label{eq:toinf}
\begin{split}
&\wt A_0=A_0,\\
&\wt A_i=A_i+\lambda^{-1}\sum_{k=0}^{i-1}\lambda^{-k}A_{i-k+1}Q,\quad i=1,\ldots,d.
\end{split}
\end{equation}
A similar expression can be obtained with the left shift.
In the transformed matrix polynomial, the eigenvalue $\lambda$ is shifted to infinity.

\subsection{Palindromic polynomials}
Consider a $*$-palindromic matrix polynomial $A(z)=\sum_{i=0}^d z^i A_i$, that is,  such that $A_i=A_{d-i}^*$  for $i=0,\ldots,d$, and observe that $A(z)^*=\bar z^dA(\bar z^{-1})$. This way, one has
 $A(\lambda)u=0$ for some $\lambda\in\mathbb C$ and some vector $u\ne 0$, if and only if $u^*A(\bar\lambda^{-1})=0$. Thus the eigenvalues of $A(z)$ come into pairs $(\lambda, \bar\lambda^{-1})$,
 where we assume that $1/0=\infty$ and $1/\infty=0$.
 
 If we are given $\lambda$ and $u$ such that $A(\lambda)u=0$, we can apply the right shift to move $\lambda$ to some value $\mu$ and the left shift to the polynomial $A(z)^*$, thus shifting $\bar\lambda$ to $\bar\mu$.
For $\mu\in\mathbb C$, the functional expression of this double shift is
\[
 \wt A(z)=\left(I+\frac{\bar\lambda-\bar\mu}{ z^{-1}-\bar\lambda}Q\right)A(z)\left(I+\frac{\lambda-\mu}{z-\lambda}Q\right),
\]
where $Q=u u^*/(u^*u)$.
We may easily check that $\wt A(z)^*=\bar z^d\wt A(\bar z^{-1})$, i.e., $\wt A(z)$ is $*$-palindromic. Moreover, one has
\[
\det 
\wt A(z)=\left(1+\frac{\bar\lambda-\bar\mu}{ z^{-1}-\bar\lambda}\right)\det A(z)\left(1+\frac{\lambda-\mu}{z-\lambda}\right)=\det A(z)\frac{(z-\mu)(1-\bar\mu z)}{(z-\lambda)(1-\bar\lambda z)},
\]
hence the pair of eigenvalues $(\lambda, \bar\lambda^{-1})$ is moved to the pair $(\mu, \bar\mu^{-1})$.
The coefficients of the shifted matrix polynomial $\wt A(z)=\sum_{i=0}^d \wt z^iA_i$
are given by
\[
\wt A_0=\wh A_0,\quad\wt A_i=\wh A_i+(\bar \lambda-\bar \mu)\sum_{k=0}^{i-1}\bar\lambda^k Q \wh A_{i-k-1},\quad i=1,\ldots,d,
\]
where
\[
\wh A_d= A_d,\quad
\wh A_j=A_j+(\lambda-\mu)\sum_{k=0}^{d-j-1}\lambda^k A_{k+j+1}Q,\quad j=0,\ldots,d-1.
\]
In the particular case where $\lambda=0$, the expression for the coefficients simplifies to
\[
\wt A_i=A_i-\mu A_{i+1}Q-\bar \mu Q  A_{i-1}+|\mu|^2QA_iQ,\quad i=0,\ldots,d,
\]
with the convention that $A_{-1}=A_{d+1}=0$.

In the specific case where $A(z)$ is a matrix polynomial of degree 2, the coefficients of $\wt A(z)$ have the simpler expression:
\[
\begin{split}
& \wt A_0=A_0+(\lambda-\mu)(A_1+\lambda A_2)Q,\\
&\wt A_1=A_1+(\lambda-\mu)A_2 Q + (\bar\lambda-\bar\mu)Q A_0+|\lambda-\mu|^2 Q(A_1+\lambda A_2)Q,\\
&\wt A_2=\wt A_0^*.
\end{split}
\]
In the above formulas, since $(A_1+\lambda A_2)Q=-\lambda^{-1}A_0 Q$ and $Q(A_1+\bar\lambda A_0)=-\bar\lambda^{-1}QA_2$, we may easily check that $\wt A_1^*=\wt A_1$.

 Concerning canonical factorizations, we may prove that if $d=2m$ and
 \[
 z^{-m}A(z)=U(z)L(z^{-1}), \quad U( z)=L(\bar z)^*,
 \]
 is a canonical factorization of $z^{-m}A(z)$, and if $|\lambda|<1, |\mu|<1$, then the shifted function $z^{-m}\wt A(z)$ has the canonical factorization 
 \[
 z^{-m}\wt A(z)=\wt U(z)\wt L(z^{-1}), \quad \wt U( z)=\wt L(\bar z)^*,\quad 
 \wt L(z^{-1})=L(z^{-1})(I+\frac{\lambda-\mu}{z-\lambda}Q).
 \]

\subsection{Quadratic matrix polynomials and matrix equations}\label{sec:quad}
Recall that \cite{blm:book} a matrix Laurent polynomial of the kind $A(z)=z^{-1}A_{-1}+A_0+zA_1$ has a canonical factorization
\[
A(z)=(I-zR_+)K_+(I-z^{-1}G_+)
\]
if and only if the matrices $R_+$ and $G_+$ are the solutions of the matrix equations
\[
X^2A_{-1}+XA_0+A_1=0,\quad A_{-1}+A_0X+A_1X^2=0,
\]
respectively, with spectral radius less than 1.
From Theorem \ref{thm:crspr} it follows that the matrices $\wt R_+=R_+$ and $\wt G_+=G_++(\lambda-\mu) Q$
 are the solutions of spectral radius less than 1 to the following matrix equations
\[
X^2\wt A_{-1}+X\wt A_0+\wt A_1=0,\quad \wt A_{-1}+\wt A_0X+\wt A_1X^2=0,
\]
respectively, where $\wt A(z)=A(z)(I+\frac{\lambda-\mu}{z-\lambda}Q)$ is the matrix Laurent polynomial obtained with the right shift with $|\lambda|,|\mu|<1$.

If in addition $A(z^{-1})$ has the canonical factorization
\[
A(z^{-1})=(I-zR_-)K_-(I-z^{-1}G_-),
\]
and the assumptions of Theorem \ref{thm:crspr} are satisfied, then  the matrices 
$\wt R_-$ and $\wt G_-$ defined in Theorem \ref{thm:crspr}
are the solution of spectral radius less than 1 to the matrix equations
 \[
X^2\wt A_{1}+X\wt A_0+\wt A_{-1}=0,\quad \wt A_{1}+\wt A_0X+\wt A_{-1}X^2=0,
\]
respectively.

For a general matrix polynomial $A(z)=\sum_{i=0}^d z^i A_i$ the existence of the canonical factorization $z^{-1}A(z)=U(z)(I-z^{-1}G)$ implies that $G$ is the solution of minimal spectral radius of the equation
\[
\sum_{i=0}^d A_iX^i=0.
\] 
Let $\lambda$ be an eigenvalue of $A(z)$ of modulus less than 1 such that $A(\lambda)u=0$, $u\in\mathbb C^n$, $u\ne 0$ and let $\mu\in\mathbb C$ be such that $|\mu|<1$. Consider the polynomial $\wt A(z)$ defined in Theorem \ref{th:matpol} obtained by means of the right shift. According to Theorem \ref{th:rightcan} the function $z^{-1}\wt A(z)$ has the canonical factorization $z^{-1}\wt A(z)=U(z)(I-z^{-1}\wt G)$ where $\wt G=G-(\lambda-\mu)Q$. Therefore the matrix $\wt G$ is the solution of minimal spectral radius of the matrix equation
\[
\sum_{i=0}^d \wt A_iX^i=0.
\]

As an example of application of this result, consider the $n\times n$ matrix polynomial $A(z)=zI-B(z)$ where $B(z)=\sum_{i=0}^dz^iB_i $ with nonnegative coefficients such that $(\sum_{i=0}^d B_i)e=e$, $e=(1,\ldots,1)^*$.
Choose $n=5$, $d=4$ and construct the coefficients $B_i$ as follows. Let $E=(e_{i,j})$, $e_{i,j}=1$, $T=(t_{i,j})$ be the upper triangular matrix such that $t_{i,j}=1$ for $i\le j$, $D=\hbox{diag}(57,49,41,33,25)$, 
set
$B_0=9D^{-1}T$, $B_1=D^{-1}T^*$, $B_2=D^{-1}E$, $B_3=D^{-1}E$, $B_4=D^{-1}$.
Then $B(1)e=e$  so that $A(1)e=0$. The matrix polynomial $A(z)$ has 5 eigenvalues of modulus less than or equal to 1, that is
\begin{verbatim}
   0.15140 - 0.01978i
   0.15140 + 0.01978i
   0.20936 - 0.09002i
   0.20936 + 0.09002i
   1.00000 + 0.00000i
\end{verbatim}
moreover, the eigenvalue of smallest modulus among those which lie outside the unit disk is {\tt 1.01258}. 
Applying cyclic reduction to approximate the minimal solution of the matrix equation $\sum_{i=0}^4A_iX^i=0$, relying on the software of \cite{bmsv}, requires 12 iterations. Applying the same method with the same software to the equation $\sum_{i=0}^4\wt A_iX^i=0$ obtained by shifting the eigenvalue 1 to zero provides the solution in just 6 iterations. In fact the approximation error at step $k$ goes to zero as $\sigma^{2^k}$ where $\sigma$ is the ratio between the eigenvalue of largest modulus, among those in the unit disk, and the eigenvalue of smallest modulus, among those out of the unit disk. This ratio is {\tt 0.98758}
for the original polynomial and {\tt 0.20676} for the polynomial obtained after the shift. This explains the difference in the number of iterations.

\section{Conclusions and open problems}
By means of a functional interpretation, we have extended Brauer's theorem to matrix (Laurent) polynomials, matrix (Laurent) power series and have related the canonical factorizations of these matrix functions to the corresponding factorizations of the shifted functions.

One point that has not been treated in this article is the analysis of the Jordan chains of a matrix polynomial or matrix power series under the right/left or double shift. This is the subject of next investigation.

Another interesting issue which deserves further study is the formal analysis of the conditioning of the eigenvalues of a matrix polynomial under the shift, together with the analysis of the conditioning of the shift operation itself and of the numerical stability of the algorithms for its implementation.

\bibliographystyle{abbrv}

\end{document}